\documentclass[psamsfonts, draft, 12pt]{amsproc}
\usepackage{amsthm}
\usepackage{mathrsfs}
\usepackage{enumerate, indentfirst}
\usepackage[fleqn,tbtags]{mathtools}
\usepackage{courier}
\usepackage[a4paper,left=2.5cm,right=2.5cm,top=3cm,bottom=3cm]{geometry}
\usepackage{hyperref}

\usepackage{xr}

\hypersetup{
    citecolor=black,%
    filecolor=black,%
    linkcolor=black,%
    pdffitwindow=true,%
    colorlinks=false,%
    unicode=true,
    }
\newtheoremstyle{mytheorem}%
{5pt}%
{3pt}%
{\itshape}%
{1pt}%
{\bf}%
{.}%
{.5em}%
{}%

\newtheoremstyle{myremark}%
{5pt}%
{3pt}%
{\upshape}%
{1pt}%
{\em}%
{.}%
{.5em}%
{}%

\newtheoremstyle{myexample}%
{5pt}%
{3pt}%
{\upshape}%
{1pt}%
{\bf}%
{.}%
{.5em}%
{}%

\theoremstyle{mytheorem}
\newtheorem{theorem}{Theorem}[section]

\newtheorem{lemma}[theorem]{Lemma}
\newtheorem{proposition}[theorem]{Proposition}
\newtheorem{corollary}[theorem]{Corollary}

\newtheorem*{theorem*}{Theorem}
\theoremstyle{myremark}

\newtheorem{remark}[theorem]{Remark}
\theoremstyle{myexample}
\newtheorem{example}[theorem]{Example}
\numberwithin{equation}{section}
\makeatletter \renewenvironment{proof}[1][\proofname] {\par\pushQED{\qed}\normalfont\topsep6\p@\@plus6\p@\relax\trivlist\item[\hskip\labelsep\itshape #1\@addpunct{.}]\ignorespaces}{\popQED\endtrivlist\@endpefalse} \makeatother

\renewcommand{\phi}{\varphi}
\renewcommand{\theta}{\vartheta}

\DeclareMathOperator{\supp}{supp}

\newcommand{\alg}{\mathscr{A}}

\newcommand{\abs}[1]{\lvert#1\rvert}
\newcommand{\dupN}{\mathbb{N}}
\newcommand{\eln}{\ell_{\dupC}^1(N)}
\newcommand{\ellegy}{\ell_{\dupC}^1(\dupN)}
\newcommand{\seq}[1]{(#1_{n})_{n\in\dupN}}

\newcommand{\dupC}{\mathbb{C}}
\newcommand{\pia}{\pi_A}
\newcommand{\dom}{\operatorname{dom}}

\newcommand{\ran}{\operatorname{ran}}

\newcommand{\beh}{\mathscr{B}(E;\hil)}
\newcommand{\bha}{\mathscr{B}(\hila)}

\newcommand{\bee}{\mathscr{B}(E;\anti{E})}
\newcommand{\bef}{\mathscr{B}(E;F)}
\newcommand{\ctc}{\mathscr{C}_0(T;\dupC)}
\newcommand{\ktc}{\mathscr{K}(T;\dupC)}

\newcommand{\D}{\mathscr{D}}
\newcommand{\M}{\mathscr{M}}

\newcommand{\hil}{H}
\newcommand{\hila}{H_A}

\DeclarePairedDelimiterX\sip[2]{(}{)}{#1\,\delimsize\vert\,#2}
\DeclarePairedDelimiterX\siptilde[2]{(}{)_{\!_{\widetilde{A}}}}{#1\,\delimsize\vert\,#2}
\DeclarePairedDelimiterX\sipn[2]{(}{)_{\nu}}{#1\,\delimsize\vert\,#2}
\DeclarePairedDelimiterX\sipm[2]{(}{)_{\mu}}{#1\,\delimsize\vert\,#2}
\DeclarePairedDelimiterX\set[2]{\{}{\}}{#1\,\delimsize\vert\,#2}
\DeclarePairedDelimiterX\dual[2]{\langle}{\rangle}{#1,#2}
\DeclarePairedDelimiterX\sipa[2]{(}{)_{\!_A}}{#1\,\delimsize\vert\,#2}
\newcommand{\anti}[1]{\bar{#1}'}
\newcommand{\bidual}[1]{\bar{#1}''}

\begin{document}
\title{Extensions of positive operators and functionals}

\author[Z. Sebesty\'en]{Zolt\'an Sebesty\'en}
\address{Z. Sebesty\'en, Department of Applied Analysis, E\"otv\"os L. University, P\'azm\'any P\'eter s\'et\'any 1/c., Budapest H-1117, Hungary; }
\email{sebesty@cs.elte.hu}
\author[Zs. Sz\H{u}cs]{Zsolt Sz\H{u}cs}
\address{Zs. Sz\H{u}cs, Department of Applied Analysis, E\"otv\"os L. University, P\'azm\'any P\'eter s\'et\'any 1/c., Budapest H-1117, Hungary;}

\email{szzsolti@cs.elte.hu}

\author[Zs. Tarcsay]{Zsigmond Tarcsay}
\address{Zs. Tarcsay, Department of Applied Analysis, E\"otv\"os L. University, P\'azm\'any P\'eter s\'et\'any 1/c., Budapest H-1117, Hungary; }
\email{tarcsay@cs.elte.hu}


\maketitle

\begin{abstract}
We consider linear operators defined on a subspace of a complex Banach space into its topological antidual acting positively in a natural sense.
The goal of this paper is to investigate of this kind of operators.
The main theorem is a constructive characterization of the bounded positive extendibility of these linear mappings.
From this result we can characterize the compactness of the extended operators and that when the positive extensions have closed ranges.

As a main application of our general extension theorem, we present some necessary and sufficient conditions that a positive functional defined on a left ideal of a Banach $^*$-algebra admits a representable positive extension.
The approach we use here is completely constructive.
\end{abstract}

\section{Introduction and preliminaries}

Positive operators and their positive extensions play a key role not only in the theory of Hilbert spaces, but in the theory of partial differential equations and mathematical physics, as well.
Hence it is worthy to examining the opportunity of positivity and positive extensions for operators defined on Banach spaces.
Since there is a canonical conjugate isometric isomorphism between a Hilbert space and its topological dual via the Riesz representation theorem, this gives us the idea that we may investigate operators between a Banach space and its \emph{topological anti\-du\-al} (see below).
It will turn out that the concept of positivity can be naturally defined for such operators, which includes the Hilbert space case, as well.

We will see in Example \ref{Ex:example4} that this kind of positive operators naturally appear in the theory of positive functionals on Banach $^*$-algebras.
Since these functionals are fundamental tools for the $^*$-representations of Banach $^*$-algebras, we will study these operators more closely later in the paper.

Before we introduce our motivations and goals, we present our terminology for positive operators.
Throughout this paper, let a complex Banach space $E$ be given.
We will denote by $\anti{E}$ the \emph{topological antidual} of $E$, that is, $\anti{E}$ consists of all continuous mappings $\phi$ of $E$ into the complex plane $\dupC$, which have the following properties:
\begin{align*}
    \phi(x+y)&=\phi(x)+\phi(y),\qquad x,y\in E,\\
    \phi(\lambda x)&=\overline{\lambda}\phi(x),\qquad x\in E,\lambda\in\dupC.
\end{align*}
The elements $\phi$ of $\anti{E}$ are called \emph{continuous anti-linear functionals} on $E$. For $x\in E$ and $\phi\in \anti{E}$ we set
\begin{align*}
    \dual{\phi}{x}:=\phi(x).
\end{align*}
It is immediately seen that $\anti{E}$ is a vector space (with pointwise operations) and that
\begin{equation*}
    \|\phi\|=\sup\set[\big]{\abs{\dual{\phi}{x}}}{x\in E, \|x\|\leq1}
\end{equation*}
defines a norm on $\anti{E}$, such that $\anti{E}$ is a Banach space with respect to this norm.
Indeed, the following canonical mapping
\begin{equation*}
    E'\to \anti{E},\quad f\mapsto \overline{f},
\end{equation*}
from  the \emph{topological dual} into the antidual of $E$ is one-to-one, onto, anti-linear and isometric with respect to the corresponding norms.
The topological antidual of $\anti{E}$, called the \emph{topological anti-bidual} of $E$ will be denoted by $\bidual{E}$ of $E$, so that $E$ can be isometrically embedded into $\bidual{E}$ along the linear mapping $j_E:E\to \bidual{E}, x\mapsto \widehat{x}$ where
\begin{equation*}
    \dual{\widehat{x}}{\phi}:=\widehat{x}(\phi)=\overline{\dual{\phi}{x}},\qquad x\in E, \phi\in \anti{E}.
\end{equation*}

If another complex Banach space $F$ is given, the \emph{(anti-)adjoint} of a continuous linear operator $T\in\bef$ is determined along the corresponding antidualities
\begin{equation}\label{E:antares}
    \dual{T^*y'}{x}=\dual{y'}{Tx}, \qquad x\in E, y'\in \anti{F},
\end{equation}
so that $T^*$ acts as a continuous linear operator between $\anti{F}$ and $\anti{E}$ with norm $\|T^*\|=\|T\|$.

Our main interest in this paper are linear operators $A$ from a linear subspace $\dom A$ of $E$ into $\anti{E}$ satisfying
\begin{equation}\label{E:positive}
    \dual{Ax}{x}\geq0,\qquad x\in\dom A.
\end{equation}
 Analogously to the case of Hilbert spaces, an operator satisfying \eqref{E:positive} will be called \emph{positive}.
\bigskip

In order to give some motivation, we have collected some examples of operators that are positive in the sense of \eqref{E:positive} (see also Examples \ref{Ex:example4} and \ref{Ex:example3}).
After these examples we introduce the main questions and purposes of the paper.

We start out with the prototype of positive operators, which justifies the usage of the word \emph{positive}:
\begin{example}\label{Ex:example1}
Let $\hil$ be a complex Hilbert space with inner product $\sip{\cdot}{\cdot}$. The topological antidual $\anti{\hil}$ can be  canonically identified with $\hil$ along the mapping $\hil\to\anti{\hil}$, $x\mapsto\sip{x}{\cdot}$, thanks to the  Riesz representation theorem.
An operator $\hil\to\hil$ therefore can be considered as an operator $\hil\to\anti{\hil}$. More precisely, if $A$ is any linear operator of $\hil$ into $\hil$ with domain $\dom A$, then we can assign an operator of $\hil$ into $\anti{H}$ (denoted also by $A$) as follows:
\begin{equation*}
    \dual{Ax}{y}:=\sip{Ax}{y},\qquad x\in\dom A, y\in \hil.
\end{equation*}
At the same time, if $A$ is positive in the classical sense (i.e., $\sip{Ax}{x}\geq0$ for $x\in\dom A$), then $A$ is positive in the sense of \eqref{E:positive}, as well.
\end{example}
A very natural construction of positive operators is described in the following example:
\begin{example}\label{Ex:example2}
    Let $E$ be a Banach space and $\hil$ be a Hilbert space with inner product $\sip{\cdot}{\cdot}$. Consider a continuous linear operator $T:E\to\hil$.
    The (anti-)adjoint $T^*$ of $T$ acts between $\anti{\hil}$ and $\anti{E}$, satisfying
    \begin{equation}\label{eq:cucu}
        \dual{T^*(\sip{x}{\cdot})}{y}=\dual{\sip{x}{\cdot}}{Ty}=\sip{x}{Ty},\qquad x\in\hil, y\in E.
    \end{equation}
    Since the mapping $x\mapsto\sip{x}{\cdot}$ is an isometric isomorphism of $\hil$ onto $\anti{\hil}$, we may consider $T^*$ as a mapping of $\hil$ into $\anti{E}$.
    Hence the composition operator $T^*T$ makes sense in such circumstances, and it is positive:
    \begin{equation*}
        \dual{T^*Ty}{y}=\sip{Ty}{Ty}\geq0,\qquad \mbox{for all $y\in E$}.
    \end{equation*}
    Furthermore, one easily verifies that the usual $C^*$-property $\|T^*T\|=\|T\|^2$ remains true also in this general setting (the proof is analogous to the Hilbert space case, \cite{Tarcsay_Banach}).
\end{example}

\begin{remark}\label{remark:exx2}
    Hereinafter we shall identify the topological antidual $\anti{\hil}$ of a Hilbert space $\hil$ with $\hil$ by the isomorphism $x\mapsto\sip{x}{\cdot}$.
In this terminology (\ref{eq:cucu}) reads: $\dual{T^*x}{y}=\sip{x}{Ty}$.
\end{remark}

The following example is self-evident.

\begin{example}\label{Ex:example5}
    Let $N$ be an arbitrary subset of $\dupN$, the set of nonnegative integers. Let us define $\eln$ by letting
    \begin{equation*}
        \eln=\set{x\in\ellegy}{\supp x\subseteq N}.
    \end{equation*}
    By considering a sequence $s\in\ell^{\infty}_{\dupC}(\dupN)$, $s(n)\geq0$ for $n\in N$, to each sequence $x\in\eln$ we can assign an element $Ax$ of the antidual of $\ellegy$ by letting
    \begin{equation*}
        \dual{Ax}{y}:=\sum_{n\in N} s(n)x(n)\overline{y(n)},\qquad y\in\ellegy.
    \end{equation*}
    By setting $E:=\ellegy$ we obtain that $A$ is a positive operator of $E$ into $\anti{E}$ with domain $\dom A=\eln$.
\end{example}

A very natural question arises in the context of positive operators: if $A$ is a positive operator from a linear subspace of the Banach space $E$ into its topological antidual $\anti{E}$, then
is there any bounded positive operator $\widetilde{A}\in\bee$ extending $A$?
If $A$ is discontinuous, there are not any, of course.
By assuming $A$ to be bounded, the answer remains henceforward no, even in the very special case when $E$ is a Hilbert space, see \cite{Krein47a}.
As we will see in our main result (Theorem \ref{T:Krein-Neumann}), the positive extendibility of $A$ is up to  the following Schwarz-type inequality (cf. \cite{Sebestyen83a,Sebestyen93} and \cite{Tarcsay_Banach} for the Banach space setting):
\begin{equation*}
    \|Ax\|^2\leq M\cdot \dual{Ax}{x}, \qquad x\in\dom A,
\end{equation*}
with some constant $M\geq0$.
Analogously to the Hilbert space case, it turns out that if the set $Ext(A)$ of all bounded positive extensions of $A$ is nonempty,  then there exists a minimal element (in the sense of a very natural partial ordering similar to the Hilbert space case, see Section \ref{S:genre}) $A_N$ of $Ext(A)$ which we will call the \emph{Krein--von Neumann extension} of $A$.
Following the treatment of \cite{Sebestyen93}, in our result Theorem \ref{T:Krein-Neumann} we will give $A_N$ via factorization over an auxiliary Hilbert space associated with $A$.
That is, we construct a Hilbert space $\hil$ and a bounded operator $T\in\beh$ such that the bounded positive operator $T^*T$ (where $T^*$ is the (anti-)adjoint of $T$, see \eqref{E:antares}) is the smallest positive extension of $A$.

Our construction makes it also possible to extend the results of \cite{Sebestyen93} and \cite{Tarcsay} on the existence of compact and closed range extensions of Hilbert space operators, as well.
Namely, we characterize those positive operators $A:E\to\anti{E}$ whose Krein--von Neumann extension is compact or has closed range (Theorems \ref{T:compuct} and \ref{T:closedrangeext}).

In the last section of our paper we apply these results in a very important special case, namely, for positive operators which are naturally induced by positive functionals on a left ideal of a Banach $^*$-algebra.
The following example describes this situation.

\begin{example}\label{Ex:example4}
    Let $\alg$ be a Banach $^*$-algebra, that is a (not necessarily unital) $^*$-algebra with a complete submultiplicative norm. Note that we do not assume that the involution of $\alg$ is continuous. Let $\M$ be any left ideal of $\alg$ and $f:\M\to \dupC$ a linear functional which is positive in the sense that
    \begin{equation*}
        f(a^*a)\geq0,\qquad \mbox{for all $a\in\M$}.
    \end{equation*}
    Assume in addition that for any fixed $a\in\M$ the following mapping
    \begin{equation*}
        Aa: \alg\to\dupC, \quad x\mapsto f(x^*a),
    \end{equation*}
    is continuous, i.e., $Aa\in\anti{\alg}$. Then clearly, $A:\alg\supseteq \M\to\anti{\alg}$ is a positive linear operator.
\end{example}

For further discussions, we recall the concept of representability of a positive linear functional.
A positive linear functional $f:\alg\to\dupC$ is called \emph{representable} if there is a representation $\pi$ of $\alg$ on a Hilbert space $\hil$ and there exists a vector $\zeta\in\hil$ such that
\begin{equation*}
    f(x)=\sip{\pi(x)\zeta}{\zeta},\qquad \mbox{for all $x\in\alg$}.
\end{equation*}
The representation $\pi$ is called (topologically) cyclic (with cyclic vector $\zeta$) if the set
\begin{equation*}
\pi\langle\alg\rangle\zeta:=\set{\pi(a)\zeta}{a\in\alg}
\end{equation*}
 is dense in the Hilbert space $\hil$.
 We note here that a representable positive functional is automatically continuous (\cite[Theorem 11.3.4]{palmer}).

It turns out that the extendibility problem of the positive operator in the example above is closely related to the existence of a representable extension of the positive functional under consideration.
Applying the procedure used for positive operators (Theorem \ref{T:Krein-Neumann}), we characterize those functionals which have any representable positive extension to the whole algebra.
We prove that in the case of representable extendibility of a functional $f$ there exists an extremal representable extension $f_N$ of $f$ which is minimal in the sense that $f_N(a^*a)\leq \widetilde{f}(a^*a)$ holds for any representable positive extension $\widetilde{f}$ of $f$, for all $a\in\alg$.
That is, our main result in this setting (Theorem \ref{T:funcext}) is the following

\begin{theorem*}
    Let $\alg$ be Banach $^*$-algebra, $\M$ a left ideal of $\alg$ and $f:\M\to\dupC$ a linear functional. The following statements are equivalent:
    \begin{enumerate}[\upshape (i)]
    \item There is a representable positive functional $f_N\in\alg'$ extending $f$, which is minimal in the sense that
          \begin{equation*}
            f_N(x^*x)\leq \widetilde{f}(x^*x),\qquad\mbox{for all $x\in\alg$,}
          \end{equation*}
          holds for any representable positive extension $\widetilde{f}\in\alg'$ of $f$.
    \item There is a representable positive functional $\widetilde{f}\in\alg'$ extending $f$.
    \item There is a number $C\geq0$ such that
    \begin{align*}
        \abs{f(a)}^2\leq C\cdot f(a^*a),\qquad \mbox{for all $a\in\M$}.
    \end{align*}
    \end{enumerate}
\end{theorem*}

Motivated by the case of positive operators, $f_N$ is called the \emph{Krein--von Neumann extension} of the functional $f$.
Our method for giving $f_N$ is constructive and is based on the GNS construction (\cite{Sebestyen84}, cf. also \cite{lesC*alg,palmer}) and on the treatment employed for giving the Krein--von Neumann extension of a positive operator.

\section{Some generalities about positive operators between Banach spaces}\label{S:genre}

In this section we investigate several general properties of positive operators from a Banach space $E$ into its antidual $\anti{E}$.
It turns out that the behavior of these operators is very similar to the well known Hilbert space operator case.
First of all we note that a positive operator $A:E\supseteq\dom A\to\anti{E}$ is automatically \emph{symmetric} in the sense that
\begin{equation}\label{E:symmetry}
    \dual{Ax}{y}=\overline{\dual{Ay}{x}},\qquad x,y\in\dom A,
\end{equation}
Indeed, an easy calculation shows this via comparing the nonnegative numbers
\begin{equation*}
\dual{A(x+y)}{x+y} \qquad \mbox{and} \qquad \dual{A(ix+y)}{ix+y}.
\end{equation*}
Thus the mapping
\begin{equation*}
    (x,y)\mapsto \dual{Ax}{y}
\end{equation*}
defines a semi-inner product on $\dom A$, and therefore the Cauchy--Schwarz inequality yields
\begin{equation}\label{E:sip-Schwarz}
    \abs{\dual{Ax}{y}}^2\leq\dual{Ax}{x}\dual{Ay}{y},\qquad x,y\in \dom A.
\end{equation}
Observe also that a partial ordering on the set of bounded positive operators can be introduced in a very natural way, namely, for two given positive operators $A, B\in\bee$ we set $A\leq B$ if and only if $\dual{Ax}{x}\leq\dual{Bx}{x}$ for all $x\in E$.

Our first result in this section is a generalization of the classical Hellinger--Toeplitz theorem which states that an everywhere defined symmetric operator on a Hilbert space is automatically continuous:
\begin{theorem}\label{T:Hellinger}
    Let $E$ be a Banach space, and let $A:E\to\anti{E}$ be a linear operator satisfying \eqref{E:symmetry}.
    Then $A$ is continuous.
\end{theorem}
\begin{proof}
    For any fixed $y\in E$, $\|y\|\leq1$, we have
    \begin{align*}
        \abs{\dual{Ay}{x}}=\abs{\dual{Ax}{y}}\leq\|Ax\|\|y\|\leq\|Ax\|, \qquad \mbox{for all $x\in E$},
    \end{align*}
    and therefore the set $\set{Ay}{y\in E,\|y\|\leq1}$ is a pointwise bounded  family of continuous anti-linear functionals on the Banach space $E$.
    In the view of  Banach's principle of uniform boundedness,
    \begin{equation*}
        \sup\set[\big]{\|Ay\|}{y\in E, \|y\|\leq1}<\infty,
    \end{equation*}
    which yields the continuity of $A$.
\end{proof}
We also note that the adjoint of an operator $A\in\bee$ acts as a continuous operator from $\bidual{E}$ into $\anti{E}$. Therefore, it makes sense to consider the composition $A^*\circ j_E$, where $j_E$ is the canonical embedding of $E$ into $\bidual{E}$.
A reasonable generalization of selfadjointness can be given therefore as follows: the operator $A\in\bee$ is called \emph{selfadjoint} if $A^*\circ j_E=A$.
In the following statement we investigate the relation between symmetry and selfadjointness in this setting:
\begin{proposition}\label{proposition:symmm}
   Let $E$ be a Banach space and $A\in\bee$ a continuous linear operator. The following statements are equivalent:
    \begin{enumerate}[\upshape (i)]
      \item $A$ is symmetric in the sense of \eqref{E:symmetry}.
      \item $A$ is selfadjoint, i.e., $A=A^*\circ j_E$.
    \end{enumerate}
\end{proposition}
\begin{proof}
    Assume first that $A$ is symmetric. Then for any $x,y\in E$ we find that
    \begin{align*}
        \dual{Ax}{y}=\overline{\dual{Ay}{x}}=\dual{j_E(x)}{Ay}=\dual{(A^*\circ j_E)(x)}{y}.
    \end{align*}
    Consequently $Ax=(A^*\circ j_E)(x)$, that is, $A$ is selfadjoint. That (ii) implies (i) is proved similarly.
\end{proof}

In the view of the above result one can reformulate Theorem \ref{T:Hellinger} as follows: an everywhere defined symmetric operator of the Banach space $E$ into $\anti{E}$ is automatically continuous and selfadjoint in the sense of Proposition \ref{proposition:symmm}.

Our last claim is to provide a generalized version of the so called operator Schwarz inequality, which plays a key role in our main result Theorem \ref{T:Krein-Neumann}.
\begin{lemma}\label{L:op-Schwarz}
    Let $E$ be a Banach space and let $A\in\bee$ be a continuous linear operator. If $A$ is positive, then
    \begin{equation}\label{E:op-Schwarz}
    \|Ax\|^2\leq\|A\|\dual{Ax}{x},\qquad\mbox{for all $x\in E$}.
\end{equation}
\end{lemma}
\begin{proof}
For any $x\in E$  the Cauchy--Schwarz inequality (\ref{E:sip-Schwarz}) yields
\begin{align*}
    \|Ax\|^2&=\sup\set[\big]{\abs{\dual{Ax}{y}}^2}{y\in E, \|y\|\leq1}\\
            &\leq \sup\set[\big]{\dual{Ax}{x}\dual{Ay}{y}}{y\in E, \|y\|\leq1}\\
            &\leq \|A\|\dual{Ax}{x},
\end{align*}
as it is claimed.
\end{proof}

\section{Main theorem on extensions of positive operators}\label{S:extension}

The following result is our main theorem in the paper.
It states that each positive operator with a Schwarz-type inequality has a smallest positive extension to the whole space, similar to the Hilbert space case (the so called \emph{Krein--von Neumann extension}, cf. \cite{Krein47a}, \cite{Krein47b}, \cite{Sebestyen83a} and \cite{Sebestyen93}).
The nontrivial part of the theorem is implication (iii)$\Rightarrow$(i).
In its proof we construct the smallest extension (see \eqref{E:pina}), and we will use the arguments of this construction in sections \ref{S:cum} and \ref{S:fuk}, as well.

\begin{theorem}\label{T:Krein-Neumann}
    Let $E$ be a complex Banach space, and let $A:E\supseteq\dom A\to \anti{E}$ be a positive operator.
    The following statements are equivalent:
    \begin{enumerate}[\upshape (i)]
      \item There exists a smallest positive extension $A_N\in \bee$ of $A$, that is, for any (bounded) positive extension $\widetilde{A}\in\bee$ of $A$ it follows that $A_N\leq \widetilde{A}$.
      \item There is a positive operator $\widetilde{A}\in \bee$ extending $A$.
      \item There is a constant $M\geq0$ such that
      \begin{equation}\label{E:extension}
        \|Ax\|^2\leq M\cdot\dual{Ax}{x},\qquad \mbox{for all $x\in\dom A.$}
      \end{equation}
    \end{enumerate}
Moreover, for any positive operator $A:E\supseteq\dom A\to\anti{E}$ which satisfies one (hence all) of the properties above, there exists a Hilbert space $\hil$ and a bounded operator $T\in\beh$ such that the bounded positive operator $T^*T$ extends $A$.
In particular: if $\dom A=E$, then $A=T^*T$.
\end{theorem}
\begin{proof}
The fact that (i) implies (ii) is obvious.

Assume that $\widetilde{A}$ is a bounded positive extension of $A$. According to Lemma \ref{L:op-Schwarz} we conclude that
\begin{align*}
    \|Ax\|^2=\|\widetilde{A}x\|^2\leq \|\widetilde{A}\|\dual{\widetilde{A}x}{x}=\|\widetilde{A}\|\dual{Ax}{x},
\end{align*}
for all $x\in\dom A$. Therefore (ii) implies (iii).

\emph{The missing implication, construction of $A_N$}:

Suppose that \eqref{E:extension} is true.
We will equip the range space $\ran A$ of $A$ with a pre-Hilbert space structure as follows: for $x,y\in\dom A$ we set
\begin{equation}\label{E:sip_A}
    \sipa{Ax}{Ay}:=\dual{Ax}{y}.
\end{equation}
First we must show that $\sipa{\cdot}{\cdot}$ is well defined: for if $Ax=Ax'$ and $Ay=Ay'$ hold for some vectors $x,x',y,y'$ from $\dom A$, then we conclude that
\begin{equation*}
    \dual{Ax}{y}=\dual{Ax'}{y}=\overline{\dual{Ay}{x'}}=\overline{\dual{Ay'}{x'}}=\dual{Ax'}{y'},
\end{equation*}
indeed. Next we show that \eqref{E:sip_A} defines an inner product.
It is readily seen from the positivity of $A$ that $\sipa{\cdot}{\cdot}$ is a semi inner product.
Furthermore, if $\dual{Ax}{x}=0$ holds for some vector $x\in\dom A$, then inequality \eqref{E:extension} gives $\|Ax\|^2=0$. Hence $\sipa{Ax}{Ax}=0$ implies $Ax=0$, as claimed.

Let us denote by $\hila$  the auxiliary Hilbert space defined as the completion of $\ran A$ equipped by the inner product \eqref{E:sip_A}.
Let $J$ stand for the natural embedding operator of $\ran A\subset H_A$ into $\anti{E}$, that is $J$ is defined by the identification
\begin{equation}\label{E:J}
    J(Ax):=Ax,\qquad x\in\dom A.
\end{equation}
Note immediately that inequality $\eqref{E:extension}$ expresses just that $J$ is continuous with norm bound $\sqrt{M}$:
\begin{equation*}
    \|J(Ax)\|^2=\|Ax\|^2\leq M\cdot\dual{Ax}{x}=M\cdot\sipa{Ax}{Ax},\qquad x\in\dom A.
\end{equation*}
Consequently, $J$ admits a unique norm preserving extension to $\hila$; we denote that operator by $J$ as well.
The (anti-)adjoint $J^*$ of $J$ therefore acts as an operator from $\bidual{E}$ into $\hila$, satisfying the following canonical extension property:
\begin{align}\label{E:J_adjoint}
    (J^*\circ j_E)(x)=Ax\in\hila,\qquad x\in\dom A.
\end{align}
Here $j_E$ stands for the natural embedding of $E$ into $\bidual{E}$.
Indeed, by taking $x$ from $\dom A$, we have for all $y\in\dom A$ that
\begin{align*}
    \sipa{(J^*\circ j_E)(x)}{Ay}=\dual{j_E(x)}{J(Ay)}=\overline{\dual{Ay}{x}}=\overline{\sipa{Ay}{Ax}}=\sipa{Ax}{Ay},
\end{align*}
whence we conclude that $(J^*\circ j_E)(x)-Ax$ is orthogonal to the dense linear manifold $\ran A$ of $\hila$, which yields \eqref{E:J_adjoint}.
Note also that $J^*\circ j_E$ acts as a bounded operator from $E$ into $\hila$, thus we conclude that its adjoint  $(J^*\circ j_E)^*$ is an operator from $\hila$ into $\anti{E}$, just as the operator $J$. We state that
\begin{equation}\label{E:J*je*=J}
    (J^*\circ j_E)^*=J.
\end{equation}
According to the continuity of the operators under consideration, in order that  \eqref{E:J*je*=J} be valid it suffices to show it on the dense set $\ran A\subseteq\hila$.
Considering a vector $y\in E$, by means of the arguments in Example \ref{Ex:example2} and Remark \ref{remark:exx2} for all $x\in\dom A$ we have that
\begin{align*}
    \dual{(J^*\circ j_E)^*(Ax)}{y}=\overline{\sipa{(J^*\circ j_E)y}{Ax}}=\overline{\dual{j_E(y)}{J(Ax)}}=\dual{J(Ax)}{y},
\end{align*}
which yields \eqref{E:J*je*=J}.
By letting
\begin{equation}\label{E:pina}
T:=(J^*\circ j_E);\ A_N:=T^*T,
\end{equation}
we conclude that the positive operator $A_N\in\bee$ extends $A$: for if $x\in\dom A$, due to identities \eqref{E:J_adjoint} and \eqref{E:J*je*=J} we infer that
\begin{equation*}
   A_Nx=T^*Tx=T^*(J^*\circ j_E)x=T^*(Ax)=J(Ax)=Ax.
\end{equation*}
Thus $A_N=T^*T$ is a positive extension, and this also proves the last statement of the theorem.

We only have to show that $A_N$ is the smallest amongst the positive extensions of $A$.
First of all observe that for each $x\in E$ we have that
    \begin{equation}\label{E:quadratic}
        \dual{A_Nx}{x}=\sup\set[\big]{\abs{\dual{Ay}{x}}^2}{y\in \dom A, \dual{Ay}{y}\leq1}.
    \end{equation}
    Indeed, by virtue of the density of $\ran A$ in the auxiliary Hilbert space $\hila$, we obtain that
    \begin{align*}
        \dual{A_Nx}{x}&=\sipa{(J^*\circ j_E)(x)}{(J^*\circ j_E)(x)}\\
                      &=\sup\set[\big]{\abs{\sipa{(J^*\circ j_E)(x)}{Ay}}^2}{y\in\dom A, \sipa{Ay}{Ay}\leq1}\\
                      &=\sup\set[\big]{\abs{\dual{j_E(x)}{J(Ay)}}^2}{y\in\dom A, \dual{Ay}{y}\leq1}\\
                      &=\sup\set[\big]{\abs{\dual{Ay}{x}}^2}{y\in \dom A, \dual{Ay}{y}\leq1}.
    \end{align*}
    Consider now a positive extension $\widetilde{A}$ of $A$.
    By repeating the construction to $\widetilde{A}$, we obtain immediately that $(\widetilde{A})_N=\widetilde{A}$, and therefore, by virtue of identity \eqref{E:quadratic}, for any $x$ from $E$ we have at once
    \begin{align*}
        \dual{\widetilde{A}x}{x}&=\sup\set[\big]{\abs{\dual{\widetilde{A}y}{x}}^2}{y\in E, \dual{\widetilde{A}y}{y}\leq1}\\
                                &\geq\sup\set[\big]{\abs{\dual{\widetilde{A}y}{x}}^2}{y\in \dom A, \dual{\widetilde{A}y}{y}\leq1}\\
                                &=\sup\set[\big]{\abs{\dual{Ay}{x}}^2}{y\in \dom A, \dual{Ay}{y}\leq1}\\
                                &=\dual{A_Nx}{x},
    \end{align*}
    which completes the proof.
\end{proof}

As we mentioned, throughout the remainder of the paper we shall make use of the construction above, and the positive extension $A_N:=T^*T$ will be referred to as the \emph{Krein--von Neumann extension} of the positive operator $A$ (analogously to the Hilbert space case).
\bigskip

We have seen in Example \ref{Ex:example2} that for a bounded operator $T$ of $E$ into a Hilbert space $\hil$ the usual $C^*$-property is valid, that is to say,
\begin{equation*}
    \|T^*T\|=\|T\|^2=\|T^*\|^2.
\end{equation*}
Therefore the norm of the Krein--von Neumann extension can be easily determined, as stated in the following

\begin{proposition}\label{P:A_N_norm}
Assume that $A$ is a positive operator of $E$ into $\anti{E}$ with domain $\dom A$ that satisfies the conditions of Theorem \ref{T:Krein-Neumann}.
The norm of the Krein--von Neumann extension $A_N$ is then calculated as follows:
\begin{equation}\label{E:A_N_norm}
    \|A_N\|=\inf\set[\big]{M\geq0}{\|Ax\|^2\leq M\cdot\dual{Ax}{x} \quad\mbox{for all $x\in\dom A$}}.
\end{equation}
\end{proposition}
\begin{proof}
    On the one hand, according to the $C^*$-property above and (\ref{E:J*je*=J}) we have that
    \begin{equation*}
        \|A_N\|=\|J^*\circ j_E\|^2=\|(J^*\circ j_E)^*\|^2=\|J\|^2.
    \end{equation*}
    On the other hand,
    \begin{align*}
        \|J\|^2&=\inf\set[\big]{M\geq0}{\|J(Ax)\|^2\leq M\cdot\sipa{Ax}{Ax}\quad\mbox{for all $x\in\dom A$}}\\
               &=\inf\set[\big]{M\geq0}{\|Ax\|^2\leq M\cdot\dual{Ax}{x} \quad\mbox{for all $x\in\dom A$}},
    \end{align*}
    as stated.
\end{proof}

\section{Compact and closed range extensions of positive operators}\label{S:cum}

Our purpose in this section is to discuss two naturally arising problems: we give necessary and sufficient conditions for a positive operator to have compact or closed range extension to the whole Banach space.
In both cases, it turns out that the existence of such kind of positive extension is equivalent with the compactness, and the closed range property of the Krein--von Neumann extension $A_N$ appeared in our Theorem \ref{T:Krein-Neumann}, respectively.

We start out by characterizing those positive operators which admit compact positive extensions.
\begin{theorem}\label{T:compuct}
    Let $A$ be a positive operator of $E$ into the topological antidual $\anti{E}$, with domain $\dom A$. The following statements are equivalent:
    \begin{enumerate}[\upshape (i)]
      \item There is a compact positive operator $\widetilde{A}\in\bee$ which extends $A$.
      \item The Krein--von Neumann extension $A_N$  of $A$ (exists and) is compact.
      \item The set
      \begin{equation}\label{E:compact}
        \set{Ax}{x\in\dom A, \dual{Ax}{x}\leq1}
      \end{equation}
      is totally bounded in $\anti{E}$.
    \end{enumerate}
\end{theorem}
\begin{proof}
    By recalling the proof of Theorem \ref{T:Krein-Neumann} \eqref{E:J}, one obtains readily that the set \eqref{E:compact} is precisely the image of the unit ball of the dense linear manifold $\ran A$ of the Hilbert space $\hila$ under the mapping $J$.
    Thus assertion (iii) expresses just that $J$ is compact.
    Since $A_N=J(J^*\circ j_E)$, we have that (iii) implies (ii). That (ii) implies (i) goes without saying. Finally, assume that $\widetilde{A}$ is a compact positive extension of $A$.
    According to Theorem \ref{T:Krein-Neumann}, we have then $A_N\leq \widetilde{A}$.
    Let us consider a bounded sequence $\seq{x}$ of $E$. Since $\widetilde{A}$ is compact, there is a subsequence $(x_{n_k})_{k\in\dupN}$ such that $(\widetilde{A}x_{n_k})_{k\in\dupN}$ converges in $\anti{E}$. Consequently, by letting $j,k\to\infty$ we conclude that
    \begin{gather*}
        \sipa[\big]{(J^*\circ j_E)(x_{n_k}-x_{n_j})}{(J^*\circ j_E)(x_{n_k}-x_{n_j})}=\dual{A_N(x_{n_k}-x_{n_j})}{x_{n_k}-x_{n_j}}\\
                                                                   \leq \dual{\widetilde{A}(x_{n_k}-x_{n_j})}{x_{n_k}-x_{n_j}}\to0,
    \end{gather*}
    whence we have that the sequence $((J^*\circ j_E)x_{n_k})_{k\in\dupN}$ converges in $\hila$. This means that $J^*\circ j_E$ is compact, and so is its adjoint $(J^*\circ j_E)^*=J$, by virtue of the Schauder theorem.
    Consequently, (i) implies (iii).
\end{proof}

We turn now to the case of closed range extensions.
Proposition \ref{proposition:symmm} states that if $A\in\bee$ is an operator with the symmetric property \eqref{E:symmetry}, then $A$ automatically satisfies $A=A^*\circ j_E$.
(In the case when $E$ is reflexive, that could also be expressed by saying that $A$ is selfadjoint, cf. \cite{matolcsi}.)
Thus we conclude the following inclusion on the ranges:
\begin{equation*}
    \ran A\subseteq\ran A^*.
\end{equation*}
Moreover, if $A\in\bee$ is positive, and $T$ is an any operator of $E$ into a Hilbert space $H$ such that $T^*T=A$, then, by the range characterization of Shmul'yan \cite{Smulian} we obtain that
\begin{align*}
    \ran A^*&=\set[\big]{z\in \anti{E}}{\exists m_z>0:\ \abs{\dual{z}{x}}\leq m_z\cdot\|Ax\|\ \forall x\in E}\\
            &=\set[\big]{z\in \anti{E}}{\exists m_z>0:\ \abs{\dual{z}{x}}\leq m_z\cdot\|T^*Tx\|\ \forall x\in E}\\
            &\subseteq \set[\big]{z\in \anti{E}}{\exists m'_z>0:\ \abs{\dual{z}{x}}\leq m'_z\cdot\|Tx\|\ \forall x\in E}\\
            &=\ran T^*.
\end{align*}
Therefore we have the following line of range inclusions:
\begin{equation}\label{E:rangeinclusions}
    \ran A\subseteq\ran A^*\subseteq\ran T^*.
\end{equation}
Note that there are no equalities in \eqref{E:rangeinclusions} in general, unless either of the ranges above is closed:
\begin{proposition}\label{P:closedrange}
    Let $A\in\bee$ be a positive operator and let $T$ be a bounded operator of $E$ into a Hilbert space $H$ such that $T^*T=A$. Let us suppose that either of the following conditions is satisfied:
    \begin{enumerate}[\upshape (a)]
      \item $A$ has closed range;
      \item $T$ has closed range.
    \end{enumerate}
    Then both range inclusions of \eqref{E:rangeinclusions} become equalities. In particular, the properties (a) and (b) are equivalent.
\end{proposition}
\begin{proof}
    First of all note that $\ker T^*=\ran T^{\perp}$ holds for any operator $T\in\beh$.
    Consequently, if $T$ has closed range,  the orthogonal decomposition $H=\ran T\oplus\ker T^*$ yields
    \begin{align*}
        \ran T^*=T^*\langle\ran T+\ker T^*\rangle=T^*\langle\ran T\rangle=\ran T^*T.
    \end{align*}
    Moreover, for any  operator $T\in\beh$ we have that $\ran T^*T$ is dense in $\ran T^*$.
    Indeed,
    \begin{align*}
        \ran T^*=T^*\langle\overline{\ran T}+\ker T^*\rangle=T^*\langle\overline{\ran T}\rangle\subseteq \overline{T^*\langle\ran T\rangle}=\overline{\ran T^*T}.
    \end{align*}
    Consequently, if $A$ has closed range, then $T^*$ has closed range too.
    In the view of the Banach closed range theorem, the range of $T$ is closed, as well.
\end{proof}

We are now in the position to give a characterization of those positive operators which admit any bounded positive extension with the closed range property; cf. also \cite{Tarcsay,Tarcsay2011,Tarcsay_Banach}:
\begin{theorem}\label{T:closedrangeext}
     Let $A$ be a positive operator of $E$ into the topological antidual $\anti{E}$, with domain $\dom A$.
     The following statements are equivalent:
    \begin{enumerate}[\upshape (i)]
      \item There is a closed range positive operator $\widetilde{A}\in\bee$ which extends $A$.
      \item The Krein--von Neumann extension $A_N$  of $A$ (exists and) has closed range.
      \item There are two nonnegative constants $M, M'\geq0$ such that
      \begin{equation}\label{E:closedrange}
        \|Ax\|^2\leq M\cdot\dual{Ax}{x}\leq M'\cdot \|Ax\|^2\qquad\mbox{for all $x\in\dom A$}.
      \end{equation}
    \end{enumerate}
\end{theorem}
\begin{proof}
   Assume first that the closed range positive operator $\widetilde{A}\in\bee$ extends $A.$ Then according to Theorem \ref{L:op-Schwarz}  the first inequality of \eqref{E:closedrange} comes true with $M=\|\widetilde{A}\|$. In order to prove the second inequality of \eqref{E:closedrange} let us consider the objects $\widetilde{H}_A$ and $\widetilde{J}$ which are associated to the positive operator $\widetilde{A}$, in accordance with the construction of the Krein--von Neumann extension  above.
   By letting $\widetilde{T}:=\widetilde{J}^*\circ j_E$ we have that $\widetilde{T}^*\widetilde{T}=\widetilde{A}$, and therefore Proposition \ref{P:closedrange} yields
   \begin{equation*}
    \ran \widetilde{T}^*=\ran \widetilde{A}^*,
   \end{equation*}
   so that Theorem 1 of Embry \cite{Embry} implies $\|\widetilde{T}^*x\|^2\leq \widetilde{M}\cdot\|\widetilde{A}x\|^2$ with a suitably chosen $\widetilde{M}\geq0$. Since $\widetilde{A}$ extends $A$, it follows that
   \begin{equation*}
    \dual{Ax}{x}=\dual{\widetilde{A}x}{x}=\|T^*x\|^2\leq \widetilde{M}\cdot\|\widetilde{A}x\|^2,
   \end{equation*}
   whence statement (iii) follows.

   Our next claim is to show that (iii) implies (ii). First of all observe that (iii) expresses precisely that for all $x\in\dom A$ the following inequalities hold:
   \begin{equation*}
    \|Ax\|^2\leq M\cdot\sipa{Ax}{Ax}\leq M'\cdot\|Ax\|^2.
   \end{equation*}
   In other words the following two mappings
   \begin{equation*}
    Ax\mapsto\|Ax\|\qquad\mbox{and}\qquad Ax\mapsto\sqrt{\sipa{Ax}{Ax}}=:\|Ax\|_{_A},\qquad x\in\dom A,
   \end{equation*}
   define equivalent norms of the linear subspace $\ran A$ of $\anti{E}$. Recall that $J$ is the closure of the operator defined by \eqref{E:J}. Therefore by \eqref{E:J*je*=J} we conclude that
   \begin{align*}
    &\ran (J^*\circ j_E)^*=\ran J\\
                         &=\set{\phi\in\anti{E}}{\exists \seq{x}\subset\dom A, \seq{Ax}~\mbox{converges in $\hila$}, Ax_n\to \phi ~\mbox{in $\anti{E}$}}\\
                         &=\set{\phi\in\anti{E}}{\exists \seq{x}\subset\dom A, \|A(x_n-x_m)\|_{_A}\to0, Ax_n\to \phi ~\mbox{in $\anti{E}$}}\\
                         &=\set{\phi\in\anti{E}}{\exists \seq{x}\subset\dom A, \|A(x_n-x_m)\|\to 0 , Ax_n\to \phi ~\mbox{in $\anti{E}$}}\\
                         &=\set{\phi\in\anti{E}}{\exists \seq{x}\subset\dom A, Ax_n\to \phi ~\mbox{in $\anti{E}$}}\\
                         &=\overline{\ran J}.
   \end{align*}
   Whence we obtain that $(J^*\circ j_E)^*$ has closed range, as well as the operator $A_N=(J^*\circ j_E)^*(J^*\circ j_E)$ by Proposition \ref{P:closedrange} and the Banach closed range theorem.
   The proof is therefore complete.
\end{proof}

As a remarkable Corollary we obtain an extension of a result due to Dixmier \cite{dixmier} stating that a bounded positive operator on a complex Hilbert space has closed range if and only if the operator and its unique positive square root have common ranges, see also \cite{fillmore} and \cite{Tarcsay2011}:
\begin{corollary}\label{C:closedrange}
    Let $E$ be Banach space, $A\in\bee$ a positive operator and $T$ a bounded operator of $E$ into a Hilbert space $\hil$ such that $A=T^*T$. The following statements are equivalent:
    \begin{enumerate}[\upshape (i)]
      \item $A$ has closed range in $\anti{E}$;
      \item $T$ has closed range in $\hil$;
      \item $T^*$ and $A$ have equal ranges in $\anti{E}$.
    \end{enumerate}
\end{corollary}
\begin{proof}
   That both (i) and (ii) imply (iii) were proved in Proposition \ref{P:closedrange}, as well as the equivalence of (i) and (ii). Therefore our only task is to prove that (iii) implies (i).  Observe first that $A_N=A$, since $A$ is everywhere defined. Consequently, for all $x\in E$ we have
   \begin{equation*}
    \sip{Tx}{Tx}=\dual{Ax}{x}=\dual{A_Nx}{x}=\sipa{(J^*\circ j_E)x}{(J^*\circ j_E)x},
   \end{equation*}
    whence
    \begin{equation}\label{E:ranJ=ranA}
      \ran A=\ran T^*=\ran (J^*\circ j_E)^*=\ran J,
    \end{equation}
     according to Theorem 1 of Embry \cite{Embry}.
     Note also that \eqref{E:ranJ=ranA} means in other words that $J$ and the restriction of $J$ to $\ran A\subseteq \hila$ have common ranges in $\anti{E}$.
     Moreover
     \begin{equation*}
        \ker J=\ker (J^*\circ j_E)^*=\ran (J^*\circ j_E)^{\perp}\subseteq\ran A^{\perp}=\{0\},
     \end{equation*}
     thus we conclude that $J$ and its restriction to $\ran A$ have common kernels, namely the trivial subspace. An easy algebraic reasoning shows that the only way this can happen is that the operators under consideration coincide, that is to say, when their domains are the same:
     \begin{equation*}
        \ran A=\hila.
     \end{equation*}
     In particular, since $\ran A\subseteq\ran(J^*\circ j_E)$, it follows that $J^*\circ j_E$ has closed range in $\hila$.
     Taking account of the Banach closed range theorem, we have at the same time that  $(J^*\circ j_E)^*$ has closed range in $\anti{E}$, and therefore that the range of $A$ is closed as well, due to \eqref{E:ranJ=ranA}.
\end{proof}

\section{Representable extensions of positive functionals}\label{S:fuk}

In this section we apply the results of the previous sections.
First we recall some facts from the Introduction.

Let us consider the positive operator $A$ of Example \ref{Ex:example4}, that is, $\M$ is supposed to be a left ideal of the Banach $^*$-algebra $\alg$ and $f:\M\to\dupC$ is a positive functional. In addition assume that there is a constant $M\geq0$ such that
 \begin{equation}\label{E:positive_functional}
        \sup\set[\big]{\abs{f(x^*a)}^2}{x\in \alg, \|x\|\leq1}\leq M\cdot f(a^*a)
 \end{equation}
holds for all $a\in\M$. Then for any fixed $a\in\M$ the mapping $Aa:\alg\to\dupC, x\mapsto f(x^*a)$ is anti-linear and continuous by norm bound $\sqrt{M\cdot f(a^*a)}$, so that
\begin{equation*}
A:\alg\supseteq\M\to\anti{\alg}, a\mapsto Aa
\end{equation*}
is a positive operator, moreover $A$ satisfies (iii) of Theorem \ref{T:Krein-Neumann}:
\begin{align*}
            \|Aa\|^2&=\sup\set[\big]{\abs{\dual{Aa}{x}}^2}{x\in\alg, \|x\|\leq1}\\
                    &=\sup\set[\big]{\abs{f(x^*a)}^2}{x\in\alg, \|x\|\leq1}\\
                    &\leq M\cdot f(a^*a)\\
                    &=M\cdot \dual{Aa}{a}.
        \end{align*}
Consequently, along the arguments used in Theorem \ref{T:Krein-Neumann} (see \eqref{E:sip_A} and \eqref{E:J}), we can associate the auxiliary Hilbert space $\hila$ and the canonical embedding operator $J$ of $\hila$ with $A$ such that $A$ admits its Krein--von Neumann extension $A_N=JJ^*\circ j_{\alg}$.

Throughout this section we are interested in investigating representable extensions of positive functionals, which are defined on a left ideal $\M$ of a Banach $^*$-algebra $\alg$. When it is not otherwise indicated, we do not assume $\alg$ to be unital, nor the involution of $\alg$ to be continuous.

\begin{lemma}\label{L:mainlemma}
    Let $\alg$ be a Banach $^*$-algebra, $\M$ a left ideal of $\alg$, and $f:\M\to\dupC$ be a linear functional satisfying $f(a^*a)\geq0$ for $a\in\M$. Then
    \begin{equation}\label{E:spect_rad}
        f(a^*x^*xa)\leq r(x^*x)\cdot f(a^*a),\qquad a\in\M, x\in\alg,
    \end{equation}
    where $r$ stands for the spectral radius function.
    If we assume in addition that
    \begin{align}\label{E:cyclic}
        \abs{f(a)}^2\leq C\cdot f(a^*a),\qquad \mbox{for all $a\in\M$},
    \end{align}
    with some $C\geq0$, then there is a constant $M\geq0$ such that \eqref{E:positive_functional}
    holds for all $a\in\M$.
\end{lemma}
\begin{proof}
    Let us consider $a\in\M$ and $x\in\alg$ such that $r(x^*x)<1$. Then, according to the square root lemma \cite{ford}, there exists a hermitian element $y\in\alg$ such that $2y-y^2=x^*x$. Since $a-ya\in\M$, we find that
    \begin{align*}
        f(a^*a)-f(a^*x^*xa)&=f(a^*a-2a^*ya+a^*y^2a)=f\big((a-ya)^*(a-ya)\big)\geq0.
    \end{align*}
    If $x$ is  an arbitrarily chosen element of $\alg$, then for each positive number $\varrho$ satisfying $\varrho^2>r(x^*x)$ we have that $r\big((\varrho^{-1}x)^*(\varrho^{-1}x)\big)<1$, whence
    \begin{align*}
     \varrho^2\cdot f(a^*a)\geq f(a^*x^*xa),
    \end{align*}
    by the above considerations. Hence $\varrho\to\sqrt{r(x^*x)}$ yields \eqref{E:spect_rad}.

    In order to prove our second statement, let us recall that by \cite[Theorem 11.1.4]{palmer} there exists a nonnegative finite constant $m(\alg)$ such that
    \begin{equation*}
        r(xx^*)^{1/2}\leq m(\alg)\|x\|, \qquad x\in\alg.
    \end{equation*}
    Consequently, by assumption \eqref{E:cyclic} we obtain that
    \begin{equation*}
    \begin{split}
        \sup\set[\big]{\abs{f(x^*a)}^2}{x\in \alg, \|x\|\leq1}&\leq\sup\set[\big]{C\cdot f(a^*xx^*a)}{x\in \alg, \|x\|\leq1}\\
        &\leq \sup\set[\big]{C r(xx^*)\cdot f(a^*a)}{x\in \alg, \|x\|\leq1}\\
        &\leq C m(\alg)^2\cdot f(a^*a)
     \end{split}	
    \end{equation*}
    for all $a\in\M$, as it is claimed.
\end{proof}
\begin{remark}
    We mention here that the constant $m(\alg)$ appearing in the proof of the previous lemma is the so called \emph{modulus of continuity}, which is defined by
    \begin{equation*}
        m(\alg):=\sup\set[\Bigg]{\frac{r(x^*x)^{1/2}}{\|x\|}}{x\in\alg, x\neq0},
    \end{equation*}
    see \cite[Definition 11.1.3]{palmer}.
\end{remark}
The following result, usually called the Gelfand--Naimark--Segal construction, is well known for (everywhere defined) positive functionals of a Banach $^*$-algebra.
We note that our construction below of the representation induced by the positive functional under consideration differs from the standard one, see e.g. \cite{Sebestyen84}. Our treatment is  based on the results of Section \ref{S:extension} concerning positive extensions of an operator.
\begin{theorem}\label{T:GNS}
    Let $\alg$ be Banach $^*$-algebra, $\M$ a left ideal of $\alg$, and $f$ a linear functional on $\M$.
    Assume that $f$ satisfies \eqref{E:cyclic} with some constant $C\geq0.$
    Then there is a cyclic representation $\pia$ of $\alg$ on the Hilbert space $\hila$ (associated with the positive operator $A$ of Example \ref{Ex:example4}) with a cyclic vector $\zeta_A$ such that
    \begin{equation}\label{E:zetaA}
        f(a)=\sipa{\pia(a)\zeta_A}{\zeta_A},\qquad \mbox{for all~ $a\in\M$}.
    \end{equation}
\end{theorem}
\begin{proof}
    For fixed $x\in\alg$ let us define $\pia(x)\in\bha$ as the continuous linear operator arising from a densely defined one as follows:
    \begin{equation}\label{E:pi_A}
        \pia(x)(Aa):=A(xa),\qquad a\in\M.
    \end{equation}
    The well-definedness and the continuity of $\pia(x)$ is guaranteed by Lemma \ref{L:mainlemma}, namely for any $a\in\M$ one has
    \begin{align*}
        \sipa{A(xa)}{A(xa)}&=\dual{A(xa)}{xa}=f(a^*x^*xa)\\ &\leq r(x^*x)\cdot f(a^*a)= r(x^*x)\cdot \sipa{Aa}{Aa}.
    \end{align*}
    It is obvious that $\pia$ is a linear mapping of $\alg$ into $\bha$.
    For $x,y\in\alg$ we have
    \begin{align*}
        \pia(x)\pia(y)(Aa)=\pia(x)(A(ya))=A(xya)=\pia(xy)(Aa),\qquad a\in\M,
    \end{align*}
    whence the multiplicativity of $\pia$ is clear.
    Furthermore, for $x\in\alg$ and $a,b\in\M$
    \begin{align*}
        \sipa{Ab}{\pia(x^*)(Aa)}&= \sipa{Ab}{A(x^*a)}=\overline{\dual{A(x^*a)}{b}}=\overline{f(b^*x^*a)}=\overline{\dual{Aa}{xb}}\\
	&=\dual{A(xb)}{a}=\sipa{A(xb)}{Aa}=\sipa{\pia(x)(Ab)}{Aa}\\
	&=\sipa{Ab}{\pia(x)^*(Aa)}
    \end{align*}
    holds, which implies $\pia(x^*)=\pia(x)^*$, that is,  $\pia$ is a representation of $\alg$.

    We are going  to prove now that $\pia$ is cyclic. First of all observe that the following mapping
    \begin{equation*}
        Aa\mapsto f(a),\qquad a\in\M,
    \end{equation*}
    defines a  continuous linear functional on the dense linear manifold $\ran A$ of $\hila$, thanks our assumptions:
    \begin{align*}
        \abs{f(a)}^2\leq C\cdot f(a^*a)=C\cdot \dual{Aa}{a}=C\cdot \sipa{Aa}{Aa}.
    \end{align*}
    The Riesz representation theorem yields  a unique representing vector $\zeta_A$ in $\hila$ satisfying
    \begin{align}\label{E:representvector}
        f(a)=\sipa{Aa}{\zeta_A}\qquad\mbox{for all $a\in\M$.}
    \end{align}
    Our next claim is to show that $\zeta_A$ satisfies
    \begin{equation}\label{E:piazeta}
        \pia(x)\zeta_A=(J^*\circ j_{\alg})(x),\qquad \mbox{for all $x\in\alg$}.
    \end{equation}
    Indeed, for fixed $x\in\alg$ and $a\in\M$ we check that
    \begin{align*}
        \sipa{Aa}{\pia(x)\zeta_A}&=\sipa{\pia(x^*)(Aa)}{\zeta_A}=\sipa{A(x^*a)}{\zeta_A}\\
        &=f(x^*a)=\dual{Aa}{x}=\dual{J(Aa)}{x}=\sipa{Aa}{(J^*\circ j_{\alg})(x)}.
    \end{align*}
    Thus we have arrived: since for $a\in\M$ we have $(J^*\circ j_{\alg})(a)=Aa$, it follows that
    \begin{equation*}
        \set{\pia(a)\zeta_A}{a\in\M}=\ran A,
    \end{equation*}
    where the latter set is a dense linear manifold of $\hila$ by definition.
    Hence
    \begin{equation*}
        f(a)=\sipa{Aa}{\zeta_A}=\sipa{(J^*\circ j_{\alg})(a)}{\zeta_A}=\sipa{\pia(a)\zeta_A}{\zeta_A},\qquad a\in\M,
    \end{equation*}
    as it is claimed. The proof is therefore complete.
\end{proof}

\begin{remark}
We note that \eqref{E:cyclic} is a reasonable assumption for $f$.
Indeed, we construct a representation of $\mathscr{A}$ by the aid of $f$, and the assumption is equivalent with the representability for everywhere defined positive functionals (see \cite[Theorem 11.3.4]{palmer}).
\end{remark}

The reader could easily give an example when the pre-Hilbert space $(\ran A,\sipa{\cdot}{\cdot})$ of the above construction is not complete, even in the case $\M=\alg$.
In other words, in order to get a Hilbert space it is necessary to make $\ran A$ complete.
The next corollary, based on the results of Section \ref{S:cum} gives a necessary and sufficient condition of $\ran A$ being a Hilbert space (with respect to the norm $\|\cdot\|_{_A}$).
\begin{corollary}
    Let $\alg$ be a Banach $^*$-algebra and let $f:\alg\to\dupC$ be a representable positive functional. The following assertions are equivalent:
    \begin{enumerate}[\upshape (i)]
      \item The pre-Hilbert space $(\ran A,\sipa{\cdot}{\cdot})$ is complete, that is, $\ran A=\hila$;
      \item There is a constant $L\geq0$ such that
      \begin{equation}\label{E:ranAHilbert}
        f(a^*a)\leq L\cdot \sup\set{\abs{f(x^*a)}^2}{x\in \alg, \|x\|\leq1},\qquad a\in\alg.
      \end{equation}
    \end{enumerate}
\end{corollary}
\begin{proof}
    Observe first that for any $a\in\alg$ the equalities
    \begin{equation*}
        \|Aa\|^2=\sup\set{\abs{(Aa)(x)}^2}{x\in \alg, \|x\|\leq1}=\sup\set{\abs{f(x^*a)}^2}{x\in \alg, \|x\|\leq1}
    \end{equation*}
    hold.
    Hence by $\dual{Aa}{a}=f(a^*a)$ the inequality \eqref{E:ranAHilbert} precisely means that $A$ fulfills (iii) of Theorem \ref{T:closedrangeext}.
    This implies that $\ran A$ is closed in $\anti{\alg}$, and therefore that $\ran A=\ran (J^*\circ j_{\alg})$ is closed too in $\hila$, thanks to the proof of Corollary \ref{C:closedrange}.
    This means that $\ran A=\hila$, as it is claimed. The converse implication is proved analogously.
\end{proof}

We are now ready to formulate our main result on extendibility of positive functionals defined on a left ideal of a Banach $^*$-algebra.
\begin{theorem}\label{T:funcext}
    Let $\alg$ be Banach $^*$-algebra, $\M$ a left ideal of $\alg$ and $f:\M\to\dupC$ a linear functional. The following statements are equivalent:
    \begin{enumerate}[\upshape (i)]
    \item There is a representable positive functional $f_N\in\alg'$ extending $f$, which is minimal in the sense that
          \begin{equation*}
            f_N(x^*x)\leq \widetilde{f}(x^*x),\qquad\mbox{for all $x\in\alg$,}
          \end{equation*}
          holds for any representable positive extension $\widetilde{f}\in\alg'$ of $f$.
    \item There is a representable positive functional $\widetilde{f}\in\alg'$ extending $f$.
    \item $f$ satisfies \eqref{E:cyclic}.
    \end{enumerate}
\end{theorem}
\begin{proof}
    It is obvious that (i) implies (ii), and that (ii) implies (iii).
    Thus our only claim is to show that (iii) implies (i).
    Let $\zeta_A\in\hila$ be the cyclic vector of $\pia$, which exists by the proof of Theorem \ref{T:GNS}. We claim that
        \begin{equation*}
            f_N:=\overline{J\zeta_A}\in\alg'
        \end{equation*}
        is the smallest representable extension of $f$, as stated in (i). Indeed, for $a\in\M$ one concludes that
        \begin{align*}
            f_N(a)=\overline{\dual{J\zeta_A}{a}}=\sipa{(J^*\circ j_{\alg})(a)}{\zeta_A}=\sipa{Aa}{\zeta_A}=f(a),
        \end{align*}
        so $f_N$ extends $f$.
	From identity \eqref{E:piazeta} it follows that
        \begin{align*}
            f_N(x)=\sipa{(J^*\circ j_{\alg})(x)}{\zeta_A}=\sipa{\pia(x)\zeta_A}{\zeta_A},
        \end{align*}
        for all $x\in\alg$.
	Thus $f_N$ is representable.
	It only remains to show that $f_N$ is extremal in the sense of (i).
	Consider a  representable positive extension $\widetilde{f}\in\alg'$ of $f$.
	Then $\widetilde{f}$ fulfills all conditions of Theorem \ref{T:GNS}, so that the objects $\widetilde{A}, \widetilde{\hil}$ and $\widetilde{J}$ can be defined as in the proof of Theorem \ref{T:Krein-Neumann}.
	Denoting the inner product of $\widetilde{\hil}$ by $\siptilde{\cdot}{\cdot}$, and using the density of $\ran \widetilde{A}$ in $\widetilde{\hil}$, we obtain for any fixed $x\in\alg$ that
        \begin{align*}
            \widetilde{f}(x^*x)&=\dual{\widetilde{A}x}{x}=\siptilde{\widetilde{A}x}{\widetilde{A}x}\\
                              &=\sup\set[\big]{\abs{\siptilde{\widetilde{A}x}{\widetilde{A}y}}^2}{y\in\alg, \siptilde{\widetilde{A}y}{\widetilde{A}y}\leq1}\\
                              &=\sup\set[\big]{\abs{\dual{\widetilde{A}y}{x}}^2}{y\in\alg, \dual{\widetilde{A}y}{y}\leq1}\\
                              &=\sup\set[\big]{\abs{\widetilde{f}(x^*y)}^2}{y\in\alg, \widetilde{f}(y^*y)\leq1}\\
                              &\geq \sup\set[\big]{\abs{f(x^*a)}^2}{a\in\M, f(a^*a)\leq1}\\
                              &=\sup\set[\big]{\abs{\dual{Aa}{x}}^2}{a\in\M, \dual{Aa}{a}\leq1}\\
                              &=\sup\set[\big]{\abs{\sipa{(J^*\circ j_{\alg})(x)}{Aa}}^2}{a\in\M, \sipa{Aa}{Aa}\leq1}\\
                              &=\sipa{(J^*\circ j_{\alg})(x)}{(J^*\circ j_{\alg})(x)}\\
                              &=\sipa{\pia(x)\zeta_A}{\pia(x)\zeta_A}\\
                              &=\sipa{\pia(x^*x)\zeta_A}{\zeta_A}\\
                              &=f_N(x^*x),
        \end{align*}
        which completes the proof.
\end{proof}
The proof of the above theorem also shows that the minimal representable extension of $f$ coincides with the conjugate of the image of the cyclic vector $\zeta_A$ under $J$, that is,
\begin{equation*}
f_N(x)=\overline{\dual{J\zeta_A}{x}},\qquad \mbox{for all $x\in\alg$}.
\end{equation*}
At the same time, the proof also provides an explicit formula for the values of $f_N$ on positive elements of $\alg$, namely,  for $x\in\alg$
\begin{equation*}
    f_N(x^*x)=\sup\set[\big]{\abs{f(x^*a)}^2}{a\in\M, f(a^*a)\leq1}.
\end{equation*}

In the next theorem we examine the special case when the left ideal of the Banach $^*$-algebra in question possesses an approximate unit in a certain sense:
\begin{theorem}\label{T:approxunit}
    Let $\alg$ be Banach $^*$-algebra, $\M$  a left ideal of $\alg$, and $f:\M\to\dupC$  a linear functional satisfying \eqref{E:cyclic}. In addition assume that there is a norm bounded net $(e_i)_{i\in I}$ of $\alg$ such that
    \begin{equation*}
        \lim_{i,I} e_ia=a,\qquad \mbox{for all $a\in\M$}.
    \end{equation*}
    Then  the net $(A_Ne_i)_{i\in I}$ converges in the norm of $\anti{\alg}$, and its limit satisfies
    \begin{equation*}
        f_N=\lim_{i,I} \overline{A_Ne_i},
    \end{equation*}
    where $f_N$ stands for  the minimal representable extension of $f$.
\end{theorem}
\begin{proof}
    For fixed $i\in I$ we have
    \begin{equation*}
        (J^*\circ j_{\alg})(e_ia)=\pia(e_i)(Aa), \qquad \mbox{for all $a\in\M$,}
    \end{equation*}
    according to \eqref{E:J_adjoint} and \eqref{E:pi_A}. By the continuity we conclude that
    \begin{align*}
        \lim_{i,I} \pia(e_i)(Aa)=\lim_{i,I} (J^*\circ j_{\alg})(e_ia)=(J^*\circ j_{\alg})(a)=Aa.
    \end{align*}
    This means that the net $(\pia(e_i))_{i\in I}$ of $\mathscr{B}(\hila)$ converges strongly to the identity operator of $\hila$ on the dense linear manifold $\ran A$.
    At the same time, each representation of a Banach $^*$-algebra is continuous.
    Hence $(\pia(e_i))_{i\in I}$ is bounded with respect to the norm of $\mathscr{B}(\hila)$.
    Consequently, by the Banach--Steinhaus theorem $(\pia(e_i))_{i\in I}$ converges strongly to the identity operator of $\hila$. In the view of \eqref{E:piazeta}
    \begin{align*}
        \lim_{i,I} (J^*\circ j_{\alg})(e_i)=\lim_{i,I} \pia(e_i)\zeta_A=\zeta_A
    \end{align*}
    holds. Thus by continuity of $J$ we have
    \begin{align*}
        \lim_{i,I} \overline{A_Ne_i}=\lim_{i,I} \overline{(JJ^*\circ j_{\alg})(e_i)}= \overline{J\zeta_A}=f_N,
    \end{align*}
    as it is claimed.
\end{proof}

\begin{corollary}\label{C:semiunital}
    Let $\alg$ be Banach $^*$-algebra, $\M$ a left ideal of $\alg$, and $f:\M\to\dupC$  a linear functional.
    If we assume in addition that there exists $e\in\alg$ such that $ea=a$ for all $a\in\M$, then the following statements are equivalent:
    \begin{enumerate}[\upshape (i)]
      \item There is a representable positive functional $f_N\in\alg'$ extending $f$, which is minimal in the sense of  Theorem \ref{T:funcext}.
      \item $f$ satisfies \eqref{E:cyclic}.
      \item $f$ satisfies \eqref{E:positive_functional}.
    \end{enumerate}
     Furthermore, if any of the above conditions is fulfilled, then $f_N=\overline{A_Ne}$.
\end{corollary}
\begin{proof}
    In the view of Theorem \ref{T:funcext}, we only need to show that (ii) follows from (iii).
    In order to get this implication, fix $a\in\M$.
    Then according to \eqref{E:positive_functional} we find that $Aa$ is an element of $\anti{\mathscr{A}}$, thus
    \begin{align*}
        \abs{f(a)}^2&=\abs{f(ea)}^2=\abs{\dual{Aa}{e^*}}^2=\abs{\dual{A_Na}{e^*}}^2\\
        &\leq \dual{A_Ne^*}{e^*}\dual{A_Na}{a}= \dual{A_Ne^*}{e^*}f(a^*a),
    \end{align*}
    holds, due to the Cauchy--Schwarz inequality.
    Finally, the equation $f_N=\overline{A_Ne}$ follows immediately from Theorem \ref{T:approxunit}.
\end{proof}
\begin{remark}
    We notice here that property  \eqref{E:positive_functional} is weaker then \eqref{E:cyclic} in general. In particular, it is well known that a continuous positive functional on a Banach $^*$-algebra is not necessarily representable, in other words, it does not fulfill \eqref{E:cyclic}.
    On the other hand, each continuous positive functional $f$ defined on the whole of $\alg$ automatically satisfies \eqref{E:positive_functional}. Indeed,  the mapping
        \begin{equation*}
            x\mapsto f(x^*a)=\overline{f(a^*x)}, \qquad x\in\alg,
        \end{equation*}
        defines a continuous, anti-linear functional on $\alg$, for any fixed $a\in\alg$. In other words, the positive operator $A$ of $\alg$ into $\anti{\alg}$, associated with $f$, is everywhere defined and hence  continuous as well in the view of Theorem \ref{T:Hellinger}. Hence we have
        \begin{align*}
            \sup\set[\big]{\abs{f(x^*a)}^2}{x\in \alg, \|x\|\leq1}\leq \|A\|\cdot f(a^*a),
        \end{align*}
        for all $a\in \alg$, thanks to the operator Schwarz inequality \eqref{E:op-Schwarz}.
\end{remark}

The following corollary is an extension of \cite[Proposition 2.1.5]{lesC*alg} to arbitrary Banach $^*$-algebras (see also \cite[Theorem 11.3.4]{palmer}).
\begin{corollary}
    Let $\alg$ be a Banach $^*$-algebra and let $f:\alg\to\mathbb{C}$ be a positive functional satisfying
    \begin{equation*}
        \abs{f(a)}^2\leq C\cdot f(a^*a),\qquad \mbox{for all $a\in\alg$},
    \end{equation*}
    with some positive constant $C$.
    Let $\widetilde{\alg}$ denote the standard unitization of $\alg$, i.e., $\widetilde{\alg}=\dupC\times\alg$, equipped with the usual operations.
    Then there is a representable positive functional $\widetilde{f}:\widetilde{\alg}\to\mathbb{C}$ which extends $f$ in the sense that
    \begin{equation*}
        \widetilde{f}((0,a))=f(a),\qquad \mbox{for all $a\in\alg$}.
    \end{equation*}
    Furthermore, the set of representable extensions of $f$ to $\widetilde{\alg}$ has a minimal element, say $f_N$, in the sense that \begin{equation*}
        f_N(\tilde{x}^*\tilde{x})\leq \widetilde{f}(\tilde{x}^*\tilde{x}), \qquad \mbox{for all $\tilde{x}\in\widetilde{\alg}$}.
    \end{equation*}
    Moreover, $f_N$ is determined on the positive elements of $\widetilde{\alg}$ by the following formula:
    \begin{equation*}
        f_N\big((\lambda,b)^*(\lambda,b)\big)=\sup\set[\big]{\abs{f(\overline{\lambda}a+b^*a)}^2}{a\in\alg, f(a^*a)\leq1},
    \end{equation*}
    for $\lambda\in\dupC$ and $b\in\alg$.
\end{corollary}
\begin{proof}
    By letting $\M:=\{0\}\times\alg$, we conclude that $g:\M\to\dupC$, given by $g((0,a)):=f(a)$ fulfills \eqref{E:cyclic}. Hence Theorem \ref{T:funcext} can be applied.
\end{proof}

It is known that any positive functional on a $C^*$-subalgebra of a $C^*$-algebra $\alg$ admits a norm preserving positive (and therefore  representable) extension to $\alg$, see \cite[Proposition 2.3.24.]{bratteli}.
Nevertheless, the proof of that statement is not constructive, namely, it makes essentially use of the Hahn--Banach extension theorem.
Our last theorem dealing with representable extendibility of positive functionals is a sharpening of this result in a special case, namely when the  $C^*$-subalgebra under consideration is in addition a $^*$-ideal of $\alg$.
We will need the following

\begin{lemma}\label{L:isometry}
    Let $\alg$ be Banach $^*$-algebra with isometric involution, $\M$ a left ideal of $\alg$ and $f:\M\to\dupC$ a linear functional satisfying
    \begin{equation*}
    \abs{f(a)}^2\leq C\cdot f(a^*a),\qquad \mbox{for all $a\in\M$}.
    \end{equation*}
 Then $f$ also satisfies
   \begin{equation*}
    \sup\set[\big]{\abs{f(x^*a)}^2}{x\in \alg, \|x\|\leq1}\leq C\cdot f(a^*a),\qquad \mbox{for all $a\in\M$},
   \end{equation*}
    with the same constant $C$.
\end{lemma}
\begin{proof}
    We proceed just as in the proof of Lemma \ref{L:mainlemma}. For any $a\in\M$ and $x\in\alg$, by Lemma \ref{L:mainlemma} we have that
    \begin{equation*}
        f(a^*x^*xa)\leq r(x^*x)\cdot f(a^*a)\leq\|x\|^2f(a^*a),
    \end{equation*}
    where the second inequality is because of the isometry of the involution. Hence  \eqref{E:cyclic} yields
    \begin{align*}
         \sup\set[\big]{\abs{f(x^*a)}^2}{x\in \alg, \|x\|\leq1}&\leq  \sup\set[\big]{C\cdot f(a^*x^*xa)}{x\in \alg, \|x\|\leq1}\\
         &\leq  \sup\set[\big]{C\|x\|^2 f(a^*a)}{x\in \alg, \|x\|\leq1}\\
         &=C\cdot f(a^*a),
    \end{align*}
    for all $a\in\M$.
\end{proof}

\begin{theorem}\label{T:Cstaralg}
    Let $\alg$ be $C^*$-algebra,  $\M$ a closed $^*$-ideal of $\alg$, and $f:\M\to\dupC$ a positive linear functional.
    Then $f$ admits a representable extension to $\alg$. Furthermore, the minimal representable extension $f_N$ of $f$ is norm preserving, that is,  $\|f_N\|=\|f\|$.
\end{theorem}
\begin{proof}
    Since $\M$ is itself a $C^*$-algebra and $f$ is a positive functional on $\M$, we conclude that $f$ is continuous and satisfies
    \begin{equation}\label{E:fnorm}
        \abs{f(a)}^2\leq\|f\| f(a^*a),\qquad a\in\M.
    \end{equation}
    Then the minimal representable extension $f_N=\overline{J\zeta_A}$ of $f$ exists, in the view of Theorem \ref{T:funcext}. Our only claim therefore is to show that $\|f_N\|=\|f\|$.
    Since \eqref{E:fnorm} yields
    \begin{gather*}
        \|\zeta_A\|^2=\sup\set[\big]{\abs{\sipa{Aa}{\zeta_A}}^2}{a\in\M,\sipa{Aa}{Aa}\leq1}\\
        =\sup\set[\big]{\abs{f(a)}^2}{a\in\M,f(a^*a)\leq1}\leq \|f\|,
    \end{gather*}
    hence due to Lemma \ref{L:isometry} and Proposition \ref{P:A_N_norm} we have $\|J\|^2=\|A_N\|\leq \|f\|$. Consequently,
    \begin{equation*}
        \|f_N\|=\|\overline{J\zeta_A}\|\leq \|J\|\|\zeta_A\|\leq \|f\|.
    \end{equation*}
    The reverse inequality is obvious (as $f_N$ is an extension of $f$).
\end{proof}
\begin{remark}\label{R:Cstaralg}
    Note that under the conditions of the previous theorem the last line of inequalities of the proof implies that
    \begin{equation*}
        \|f\|=\|\zeta_A\|^2=\|A_N\|=\|f_N\|.
    \end{equation*}
    If we assume in addition the existence of an element $e\in\alg$ with $\|e\|\leq1$ such that $ea=a$ for all $a\in\M$, then, in the view of Corollary \ref{C:semiunital} we find that $f_N=\overline{A_Ne}$.
    In particular, the Krein--von Neumann extension $A_N$ of $A$ attains  its norm on $e$.
\end{remark}

In the following important special case we make use of the results of the paper.

\begin{example}\label{Ex:example3}
    Let $\Omega$ be a locally compact Hausdorff space.
    Denote the $C^*$-algebra of complex valued continuous functions vanishing at infinity by $\ctc$, and the subspace of complex valued continuous functions with compact support by $\ktc$.
    Fix a positive Radon measure $\mu$ on $\Omega$, that is to say, $\mu$ is a positive linear functional on $\ktc$. Consider a relatively compact subset $K$ of $\Omega$.
    By setting
    \begin{equation}\label{E:D}
        \D(K):=\set{f\in\ctc}{\supp f\subseteq K},
    \end{equation}
    one easily verifies that  $\D(K)$ is an ideal (moreover, a $^*$-ideal) of $\ctc$. To a given $f\in \D(K)$ we can assign an element $Af$ of  the topological antidual of $\ctc$, namely
    \begin{equation}\label{E:Af}
        \dual{Af}{g}:=\mu(f\cdot \overline{g}),\qquad g\in\ctc.
    \end{equation}
    Since $f$ has compact support, one obtains that $Af$ is continuous with $\|Af\|=\mu(\abs{f})$, and therefore that $A$ is a positive operator with domain $\D(K)$.

    Denote by $\mu_{_K}$ the restriction of $\mu$  to $\D(K)$.
    Fix a nonnegative valued continuous function $e_{_K}$ with compact support such that $0\leq e_{_K}\leq1$ and that
    \begin{equation}\label{E:suppfk}
        K\subseteq \set{x\in \Omega}{e_{_K}(x)=1}.
    \end{equation}
    Then for all $f\in\D(K)$ we conclude that
    \begin{align*}
        \abs{\mu_{_K}(f)}^2 =\mu(\abs{e_{_K}\!\cdot\! f})^2\leq \mu(e_{_K}^2)\mu(\abs{f}^2)=\mu(e_{_K}^2)\mu_{_K}(\abs{f}^2).
    \end{align*}
    Theorem \ref{T:funcext} implies that $\mu_{_K}$ has a minimal representable extension to $\ctc$, say $\mu_{_{K,N}}$.
    By continuity $\mu_{_K}$  can be uniquely extended to the norm closure of $\D(K)$, which is a closed $^*$-ideal of the $C^*$-algebra $\ctc$. Denoting the unique continuous extension also with $\mu_{_K}$, it can be easily verified that $\mu_{_K}$ fulfills all conditions of Theorem \ref{T:Cstaralg}. Thus we find that $\|\mu_{_K}\|=\|\mu_{_{K,N}}\|$. Furthermore, for any $f\in\D(K)$ we have $\|Af\|=\mu(\abs{f})=\mu_{_K}(\abs{f})$, and therefore that
    \begin{align*}
        \|A\|&=\sup\set[\big]{\|Af\|}{f\in\D(K), \|f\|\leq1}\\
             &=\sup\set[\big]{\mu_{_K}(\abs{f})}{f\in\D(K), \|f\|\leq1}\\
             &=\sup\set[\big]{\abs{\mu_{_K}(f)}}{f\in\D(K), \|f\|\leq1}\\
             &=\|\mu_{_K}\|.
    \end{align*}
    Since $\|e_{_K}\|\leq1$ and for any $g\in\overline{\D(K)}$ the equation $e_{_K}\cdot g=g$ holds, we have that $\mu_{_{K,N}}=\overline{A_Ne_{_K}}$ and that $A_N$ is a norm preserving extension of $A$ thanks to Remark \ref{R:Cstaralg}:
    \begin{equation*}
        \|A\|=\|\mu_{_K}\|=\|\mu_{_{K,N}}\|=\|A_N\|.
    \end{equation*}
    In particular this implies
    \begin{equation*}
    \begin{split}
        &\sup\set[\big]{\mu(\abs{f})}{f\in\D(K), \abs{f}\leq1}\\
        &=\inf\set[\big]{M\geq0}{\mu(\abs{f})^2\leq M\cdot\mu(\abs{f}^2)\quad\mbox{for all $f\in\D(K)$}},
    \end{split}
    \end{equation*}
    thanks to Proposition \ref{P:A_N_norm}.

    We notice here that to each nonnegative function $e_{_K}\in\ktc$, satisfying \eqref{E:suppfk}, we can associate a positive functional on $\ctc$ denoted by $e_{_K}.\mu$, which extends $\mu_{_K}$:
    \begin{equation*}
        (e_{_K}.\mu)(g):=\mu(e_{_K}\cdot g), \qquad g\in\ctc.
    \end{equation*}
    It is readily seen that $e_{_K}.\mu$ extends $\mu_{_K}$, and, as it is well known, $\|e_{_K}.\mu\|=\mu(e_{_K})$.
    Furthermore, we have in general
    \begin{equation*}
        \sup\set[\big]{\abs{\mu(f)}}{f\in\D(K), \abs{f}\leq1}<\mu(e_{_K}),
    \end{equation*}
    for each $e_{_K}$ fulfilling \eqref{E:suppfk}. That means in particular that there is no $e_{_K}$ such that $e_{_K}.\mu=\mu_{_{K,N}}$.
    \end{example}

We close our paper with an example of a positive functional which does not admit any representable extension:
\begin{example}
    Let $\alg$ denote the unital $C^*$-algebra of all $2$-by-$2$ complex matrices. By setting
    \begin{equation*}
        \M:=\set[\bigg]{\left(\begin{array}{cc}
                          u & 0 \\
                          v & 0
                        \end{array}\right)
        }{u,v\in\dupC},
    \end{equation*}
    one easily verifies that $\M$ is a left ideal of $\alg$, and that the mapping
    \begin{equation*}
        f:\M\to\dupC,\qquad \left(\begin{array}{cc}
                          u & 0 \\
                          v & 0
                        \end{array}\right) \mapsto v,
    \end{equation*}
    defines a linear functional, satisfying $f(a^*a)=0$ for all $a\in\M$.
    It is readily seen therefore that $f$ is a positive functional which is not satisfy the conditions (ii) and (iii) of Corollary \ref{C:semiunital}.
    Consequently Corollary \ref{C:semiunital} implies that $f$ does not have any representable extension.
\end{example}

\end{document}